\documentclass[11pt]{article}
\usepackage[utf8]{inputenc}
\usepackage[a4paper, margin = 1in]{geometry}
\usepackage{amsmath,amssymb,amsthm,tikz,tikz-cd,color,xcolor,ytableau,mathdots,amscd}
\usepackage{graphicx}
\usepackage{float}
\usepackage{physics}
\usepackage{relsize}
\newtheorem{theorem}{Theorem}

\newtheorem{lemma}[theorem]{Lemma}
\theoremstyle{definition}

\definecolor{darkgreen}{RGB}{0,180,0}
\colorlet[named]{green}{darkgreen}

\newcommand{\End}{\operatorname{End}}

\newcommand{\gammaice}[4]{\begin{tikzpicture}
\coordinate (a) at (-.75, 0);
\coordinate (b) at (0, .75);
\coordinate (c) at (.75, 0);
\coordinate (d) at (0, -.75);
\coordinate (aa) at (-.75,.5);
\coordinate (cc) at (.75,.5);
\draw (a)--(c);
\draw (b)--(d);
\draw[fill=white] (a) circle (.25);
\draw[fill=white] (b) circle (.25);
\draw[fill=white] (c) circle (.25);
\draw[fill=white] (d) circle (.25);
\node at (0,1) { };
\node at (a) {$#1$};
\node at (b) {$#2$};
\node at (c) {$#3$};
\node at (d) {$#4$};
\path[fill=white] (0,0) circle (.2);
\node at (0,0) {$\bullet$};
\end{tikzpicture}}
\newcommand{\gammagamma}[4]{
\begin{tikzpicture}[scale=0.7]
\draw (0,0) to [out = 0, in = 180] (2,2);
\draw (0,2) to [out = 0, in = 180] (2,0);
\draw[fill=white] (0,0) circle (.35);
\draw[fill=white] (0,2) circle (.35);
\draw[fill=white] (2,0) circle (.35);
\draw[fill=white] (2,2) circle (.35);
\node at (0,0) {$#1$};
\node at (0,2) {$#2$};
\node at (2,2) {$#3$};
\node at (2,0) {$#4$};
\path[fill=white] (1,1) circle (.3);
\node at (1,1) {$\bullet$};
\end{tikzpicture}
}

\title{Lattice Models, Differential Forms, and the Yang-Baxter Equation}
\date{December 22, 2020}
\author{Kedar Karhadkar}
\begin{document}
\maketitle
\begin{abstract}
    We introduce new methods to describe admissible states of the six-vertex and the eight-vertex lattice models of statistical mechanics. For the six-vertex model, we view the admissible states as differential forms on a grid graph. This yields a new proof of the correspondence between admissible states and 3-colorings of a rectangular grid. For the eight-vertex model, we interpret the set of admissible states as an $\mathbb{F}_2$-vector space. This viewpoint lets us enumerate the set of admissible states.  Finally, we find necessary conditions for a Yang-Baxter equation to hold for the general eight-vertex model.
\end{abstract}
\section{Introduction}
Lattice models are objects from statistical mechanics which have also seen a number of surprising applications in mathematics. These models are built with a graph (often a rectangular grid) whose edges are labeled, and locally satisfy some property. If this property is satisfied, we call the state \emph{admissible}. Each vertex of the graph is assigned a weight depending on the labels of the edges around it. The goal is to make conclusions about the global behavior of the graph. One way of studying this is by determining the \emph{partition function} of the lattice model, which is calculated by taking a product of the vertex weights, and summing over all possible states:
\[Z(\mathfrak{G}) = \sum_{\text{states}}\prod_{v\in V} \text{wt}(v). \]
Baxter \cite{Baxter} studied various models in this way. An important method that he used was repeated application of the \emph{Yang-Baxter equation}, from which he deduced certain symmetry properties of partition functions. These symmetries are the source of several connections to representation theory. For instance, if we assign the vertices a particular set of weights, then the partition function we obtain is a Schur function \cite{BBF}. This allows one to prove symmetric function identities using lattice models, such as the dual Cauchy identity \cite{DualCauchy} and the Weyl character formula \cite{WeylCharacter}.

We will primarily focus our attention on the six-vertex and the eight-vertex square lattice models. The six-vertex model was famously used by Kuperberg \cite{Kuperberg} to prove the alternating sign matrix conjecture, originally proven by Zeilberger \cite{Zeilberger}. Among the first steps of the proof was a correspondence between admissible lattice states of the six vertex model and alternating sign matrices. In this paper, we will find combinatorial interpretations of lattice states of a similar flavor.

The primary technique that we will use to study the states of lattice models will be the use of discrete differential forms on a rectangular grid. Discrete differential calculus has been studied recently over arbitrary graphs; for a more general description of it, see \cite{DifferentialForms}. Our methods rely on the observation that for certain edge values $f_{i,j}$ and $g_{i,j}$, we have
\[g_{i+1,j} - g_{i,j} = f_{i,j+1} - f_{i,j}.\]
We rewrite this as
\[D_xg = D_yf,\]
where $D_x$ and $D_y$ are \emph{discrete partial derivatives}. Drawing an analogy to the continuous case, we say that the 1-form $fdx + gdy$ is \emph{closed}. Since rectangular grids are discrete analogues of open balls in $\mathbb{R}^2$, we see that closed 1-forms on the rectangular grid are exact. This gives us a new proof that there are three times as many 3-colorings of a rectangular grid as there are admissible lattice states in the six-vertex model, a fact proven by Lenard \cite{3color}.

Our viewpoint becomes more fruitful when we look at more complicated lattices. Specifically, we study the six-vertex model with toroidal boundary conditions. This is a discrete analogue of a torus, so we expect it to have nontrivial cohomology. By computing this cohomology, we are able to find a new combinatorial interpretation of the admissible states with toroidal boundary conditions. In particular, we see that the number of admissible states are not simply in correspondence with 3-colorings of a toroidal grid graph.

After studying the six-vertex model, we turn our attention to admissible states of the eight-vertex model. Although we could use our previous differential forms strategy, we find that the structure of the admissible states is even simpler. In particular, there is a natural description of the admissible states as a vector space over $\mathbb{F}_2$. This lets us explicitly determine the number of admissible states in the eight-vertex model. Moreover, we can determine this number for any given set of boundary conditions, and it is essentially independent of the boundary.

Finally, we turn our attention back to the Yang-Baxter equation in the eight-vertex model. Partition functions for the eight-vertex model have been studied by Fan and Wu \cite{FanWu} and Galleas and Martins \cite{GalleasMartins}. In their analysis, they make the assumption that for certain parameters $c_1, c_{-1}, d_1, d_{-1}$, we have $c_1 = c_{-1}$ and $d_1 = d_{-1}$. We prove our results while relaxing these assumptions, so we are able to generalize their previous work. We determine conditions the Boltzmann weights of the lattice that are required in order to find a solution to the Yang-Baxter equation.
\section{Admissible States of the Six-Vertex Model}
Here we will quickly define the six-vertex lattice model. Let us consider a rectangular grid in which the edges are labeled as either 0 or 1. We will refer to such a labeling of the edges as a \emph{state}.

\begin{figure}[h]
\[
  \scalebox{.95}{\begin{tikzpicture}
    \draw [line width=0.45mm] (1,0)--(1,6);
    \draw [line width=0.45mm] (3,0)--(3,6);
    \draw [line width=0.45mm] (5,0)--(5,6);
    \draw [line width=0.45mm] (7,0)--(7,6);
    \draw [line width=0.45mm] (9,0)--(9,6);
    \draw [line width=0.45mm] (0,1)--(10,1);
    \draw [line width=0.45mm] (0,3)--(10,3);
    \draw [line width=0.45mm] (0,5)--(10,5);
    \draw[line width=0.45mm, fill=white] (1,6) circle (.35);
    \draw[line width=0.45mm, fill=white] (3,6) circle (.35);
    \draw[line width=0.45mm, fill=white] (5,6) circle (.35);
    \draw[line width=0.45mm, fill=white] (7,6) circle (.35);
    \draw[line width=0.45mm, fill=white] (9,6) circle (.35);
    \draw[line width=0.45mm, fill=white] (1,4) circle (.35);
    \draw[line width=0.45mm, fill=white] (3,4) circle (.35);
    \draw[line width=0.45mm, fill=white] (5,4) circle (.35);
    \draw[line width=0.45mm, fill=white] (7,4) circle (.35);
    \draw[line width=0.45mm, fill=white] (9,4) circle (.35);
    \draw[line width=0.45mm, fill=white] (1,2) circle (.35);
    \draw[line width=0.45mm, fill=white] (3,2) circle (.35);
    \draw[line width=0.45mm, fill=white] (5,2) circle (.35);
    \draw[line width=0.45mm, fill=white] (7,2) circle (.35);
    \draw[line width=0.45mm, fill=white] (9,2) circle (.35);
    \draw[line width=0.45mm, fill=white] (1,0) circle (.35);
    \draw[line width=0.45mm, fill=white] (3,0) circle (.35);
    \draw[line width=0.45mm, fill=white] (5,0) circle (.35);
    \draw[line width=0.45mm, fill=white] (7,0) circle (.35);
    \draw[line width=0.45mm, fill=white] (9,0) circle (.35);
    \draw[line width=0.45mm, fill=white] (0,5) circle (.35);
    \draw[line width=0.45mm, fill=white] (2,5) circle (.35);
    \draw[line width=0.45mm, fill=white] (4,5) circle (.35);
    \draw[line width=0.45mm, fill=white] (6,5) circle (.35);
    \draw[line width=0.45mm, fill=white] (8,5) circle (.35);
    \draw[line width=0.45mm, fill=white] (10,5) circle (.35);
    \draw[line width=0.45mm, fill=white] (0,3) circle (.35);
    \draw[line width=0.45mm, fill=white] (2,3) circle (.35);
    \draw[line width=0.45mm, fill=white] (4,3) circle (.35);
    \draw[line width=0.45mm, fill=white] (6,3) circle (.35);
    \draw[line width=0.45mm, fill=white] (8,3) circle (.35);
    \draw[line width=0.45mm, fill=white] (10,3) circle (.35);
    \draw[line width=0.45mm, fill=white] (0,1) circle (.35);
    \draw[line width=0.45mm, fill=white] (2,1) circle (.35);
    \draw[line width=0.45mm, fill=white] (4,1) circle (.35);
    \draw[line width=0.45mm, fill=white] (6,1) circle (.35);
    \draw[line width=0.45mm, fill=white] (8,1) circle (.35);
    \draw[line width=0.45mm, fill=white] (10,1) circle (.35);

    \path[fill=white] (1,1) circle (.2);
    \node at (1,1) {$\bullet$};
    \path[fill=white] (3,1) circle (.2);
    \node at (3,1) {$\bullet$};
    \path[fill=white] (5,1) circle (.2);
    \node at (5,1) {$\bullet$};
    \path[fill=white] (7,1) circle (.2);
    \node at (7,1) {$\bullet$};
    \path[fill=white] (9,1) circle (.2);
    \node at (9,1) {$\bullet$};

    \path[fill=white] (1,3) circle (.2);
    \node at (1,3) {$\bullet$};
    \path[fill=white] (3,3) circle (.2);
    \node at (3,3) {$\bullet$};
    \path[fill=white] (5,3) circle (.2);
    \node at (5,3) {$\bullet$};
    \path[fill=white] (7,3) circle (.2);
    \node at (7,3) {$\bullet$};
    \path[fill=white] (9,3) circle (.2);
    \node at (9,3) {$\bullet$};

    \path[fill=white] (1,5) circle (.2);
    \node at (1,5) {$\bullet$};
    \path[fill=white] (3,5) circle (.2);
    \node at (3,5) {$\bullet$};
    \path[fill=white] (5,5) circle (.2);
    \node at (5,5) {$\bullet$};
    \path[fill=white] (7,5) circle (.2);
    \node at (7,5) {$\bullet$};
    \path[fill=white] (9,5) circle (.2);
    \node at (9,5) {$\bullet$};

    \node at (1,6) {$1$};
    \node at (3,6) {$0$};
    \node at (5,6) {$1$};
    \node at (7,6) {$0$};
    \node at (9,6) {$1$};
    \node at (1,4) {$0$};
    \node at (3,4) {$0$};
    \node at (5,4) {$1$};
    \node at (7,4) {$1$};
    \node at (9,4) {$0$};
    \node at (1,2) {$0$};
    \node at (3,2) {$0$};
    \node at (5,2) {$0$};
    \node at (7,2) {$1$};
    \node at (9,2) {$0$};
    \node at (1,0) {$0$};
    \node at (3,0) {$0$};
    \node at (5,0) {$0$};
    \node at (7,0) {$0$};
    \node at (9,0) {$0$};
    \node at (0,5) {$0$};
    \node at (2,5) {$1$};
    \node at (4,5) {$1$};
    \node at (6,5) {$1$};
    \node at (8,5) {$0$};
    \node at (10,5) {$1$};
    \node at (0,3) {$0$};
    \node at (2,3) {$0$};
    \node at (4,3) {$0$};
    \node at (6,3) {$1$};
    \node at (8,3) {$1$};
    \node at (10,3) {$1$};
    \node at (0,1) {$0$};
    \node at (2,1) {$0$};
    \node at (4,1) {$0$};
    \node at (6,1) {$0$};
    \node at (8,1) {$1$};
    \node at (10,1) {$1$};
    \node at (1.00,6.8) {$ 1$};
    \node at (3.00,6.8) {$ 2$};
    \node at (5.00,6.8) {$ 3$};
    \node at (7.00,6.8) {$ 4$};
    \node at (9.00,6.8) {$ 5$};

    \node at (-.75,1) {$ 1$};
    \node at (-.75,3) {$ 2$};
    \node at (-.75,5) {$ 3$};
  \end{tikzpicture}}
\]
\caption{A state of a rectangular lattice.}
\end{figure}
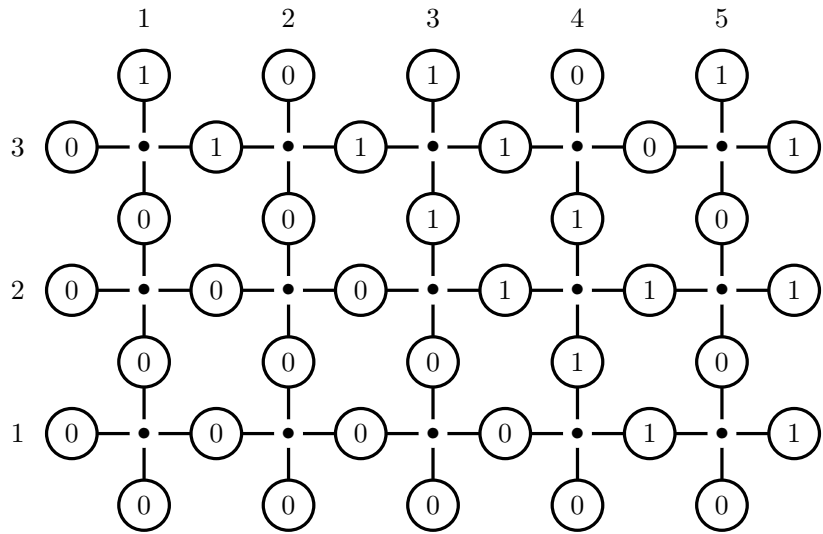
We are particularly interested in the states in which the edges around a vertex are labeled in the following way.

\begin{figure}[H]
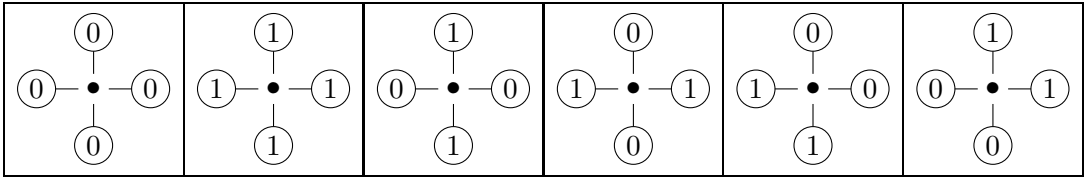

\[
\begin{array}{|c|c|c|c|c|c|}
\hline
  \gammaice{0}{0}{0}{0} &
  \gammaice{1}{1}{1}{1} &
  \gammaice{0}{1}{0}{1} &
  \gammaice{1}{0}{1}{0} &
  \gammaice{1}{0}{0}{1} &
  \gammaice{0}{1}{1}{0}\\
\hline\end{array}\]
\caption{Admissible labelings in the six-vertex model.}
\label{uncoloredbw}
\end{figure}

If all of the vertices of a state are in one of the six configurations of Figure 2, then we call the state \emph{admissible}. For instance, the state in Figure 1 is admissible.

For the rest of this section, as well as sections 3 and 4, fix $m, n \geq 2$. For a rectangular lattice with $m$ columns and $n$ rows of interior vertices, let $v_{i,j}$ denote the vertex that is in the $i$-th column from the left and the $j$-th row from the bottom. Let $g_{i,j}$ be the label of the horizontal edge which is $i$-th from the left and $j$-th from the bottom. Similarly, let $f_{i,j}$ be the label of the vertical edge which is $i$-th from the left and $j$-th from the bottom. For $1 \leq i \leq m$ and $1 \leq j \leq n$, we call the entries $g_{1,j}, g_{m+1,j}, f_{i,1}, f_{i,n+1}$ the \emph{boundary values} of the lattice. At any vertex $v_{i,j}$, we have the following labeling.
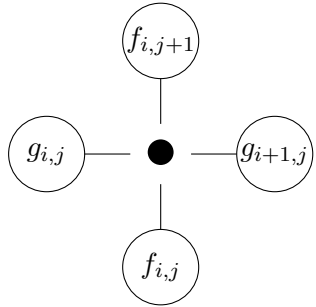
\begin{figure}[H]
\[\begin{tikzpicture}[scale=2]
\coordinate (a) at (-.75, 0);
\coordinate (b) at (0, .75);
\coordinate (c) at (.75, 0);
\coordinate (d) at (0, -.75);
\coordinate (aa) at (-.75,.5);
\coordinate (cc) at (.75,.5);
\draw (a)--(c);
\draw (b)--(d);
\draw[fill=white] (a) circle (.25);
\draw[fill=white] (b) circle (.25);
\draw[fill=white] (c) circle (.25);
\draw[fill=white] (d) circle (.25);
\node at (0,1) { };
\node at (a) {$g_{i,j}$};
\node at (b) {$f_{i,j+1}$};
\node at (c) {$g_{i+1,j}$};
\node at (d) {$f_{i,j}$};
\path[fill=white] (0,0) circle (.2);
\node at (0,0) {\Huge $\bullet$};
\end{tikzpicture}\]
\end{figure}
It is easy to verify that this labeling is admissible if and only if
\begin{align}
g_{i+1,j} - g_{i,j} \equiv f_{i,j+1} - f_{i,j} \pmod 3. \label{admissible-condition}
\end{align}
In light of this, we will view $f$ and $g$ as functions from $[m] \times [n + 1]$ to $\mathbb{F}_3$ and $[m + 1] \times [n]$ to $\mathbb{F}_3$ respectively, where $[n]$ denotes the set $\{1,2,\cdots,n\}$.

 Let $h: [m+1] \times [n+1] \to \mathbb{F}_3$ be a function. We define its \emph{discrete partial derivatives} $D_xh: [m] \times [n+1] \to \mathbb{F}_3$ and $D_yh: [m+1] \times [n] \to \mathbb{F}_3$ by
\[(D_x h)_{i,j} = h_{i+1,j} - h_{i,j} \]
and
\[(D_yh)_{i,j} = h_{i,j+1} - h_{i,j}. \]
We define a (discrete) \emph{1-form} to be a formal expression
\[fdx + gdy,\]
where $f: [m] \times [n+1] \to \mathbb{F}_3$ and $g: [m+1] \times [n] \to \mathbb{F}_3$ are functions. If $h: [m+1] \times [n+1] \to \mathbb{F}_3$ is a function, we define its \emph{exterior derivative} $dh$ by
\[dh = (D_xh)dx + (D_yh)dy. \]
We say that a 1-form $fdx + gdy$ is \emph{closed} if
\[D_y f = D_x g, \]
and \emph{exact} if $fdx + gdy = dh$ for some function $h$. It is easy to see that every exact 1-form is closed. In fact, by the following discrete version of the Poincar\'{e} lemma, the converse is also true.
\begin{lemma}\label{poincare}
Let $\alpha = fdx + gdy$ be a closed 1-form. Then $\alpha$ is exact.
\end{lemma}
\begin{proof}
Define $h: [m+1] \times [n+1] \to \mathbb{F}_3$ by
    \[h_{i,j} = \sum_{a = 1}^{i-1} f_{a,1} + \sum_{b = 1}^{j-1} g_{i,b}.\]
Then,
\begin{align*}
    (D_yh)_{i,j} &= \left(\sum_{a = 1}^{i-1} f_{a,1} + \sum_{b = 1}^{j} g_{i,b}\right) - \left(\sum_{a = 1}^{i} f_{a,1} + \sum_{b = 1}^{j-1} g_{i,b}\right)\\
    &= g_{i,j}
\end{align*}
and
\begin{align*}
    (D_xh)_{i,j} &= \left(\sum_{a = 1}^{i} f_{a,1} + \sum_{b = 1}^{j-1} g_{i+1,b}\right) - \left(\sum_{a = 1}^{i-1} f_{a,1} + \sum_{b = 1}^{j-1} g_{i,b}\right) \\
    &= f_{i,1} + \sum_{b = 1}^{j - 1} (D_xg)_{i,b}\\
    &= f_{i,1} + \sum_{b = 1}^{j - 1} (D_y f)_{i,b}\\
    &= f_{i,j}.
\end{align*}
So $dh = \alpha$, and $\alpha$ is exact.
\end{proof}
Call a 1-form $fdx + gdy$ \emph{admissible} if for all $i$ and $j$, $f_{i,j}$ and $g_{i,j}$ are not equal to 2. Using this language, we can describe admissible states in terms of differential forms.
\begin{lemma}
There is a one-to-one correspondence between admissible states of the six-vertex model and admissible closed 1-forms.
\end{lemma}
\begin{proof}
Consider an admissible state $A$ of a lattice with with $m$ columns and $n$ rows. Let $g_{i,j}$ and $f_{i,j}$ be the entries of the horizontal and vertical edges as above. Then for all $i$ and $j$, $g_{i,j}$ and $f_{i,j}$ are not equal to 2. Equation (\ref{admissible-condition}) holds exactly when
\[D_y f = D_xg,\]
so we associate $A$ with the admissible closed 1-form $fdx + gdy$.
\end{proof}
As an application of the above correspondence, we will describe admissible states in terms of colorings of a rectangular grid. For the original proof of this result, see \cite{3color}. Recall that a $k$-coloring of a graph is an assignment of values from the set $\{0,1,\cdots,k-1\}$ to each vertex of the graph, such that two adjacent vertices are not assigned the same value.
\begin{theorem}
There are three times as many 3-colorings of a rectangular grid with $m+1$ columns and $n+1$ rows as there are admissible states of the six-vertex model.
\end{theorem}
\begin{proof}
Let $\mathcal{S}_1$ denote the set of admissible closed 1-forms, and let $\mathcal{S}_2$ denote the set of functions $h: [m+1] \times [n+1] \to \mathbb{F}_3$ such that for all $i$ and $j$, $h_{i+1,j} \neq h_{i,j}$ and $h_{i,j+1} \neq h_{i,j}$. In other words, $\mathcal{S}_2$ is the set of functions $h$ such that $D_xh$ and $D_yh$ are nonzero everywhere. It is easy to see that $\mathcal{S}_2$ is in bijection with the set of 3-colorings of the rectangular grid. To prove the theorem, then, it suffices to find a bijection between $\mathcal{S}_1 \times \mathbb{F}_3$ and $\mathcal{S}_2$. By Lemma \ref{poincare}, every element of $\mathcal{S}_1$ is exact, so we can write every element of $\mathcal{S}_1$ in the form $dh$, where $h:[m+1]\times[n+1]\to\mathbb{F}_3$ is a function. Now define $F: \mathcal{S}_1 \times \mathbb{F}_3\to \mathcal{S}_2$ by
\[F(dh, t)_{i,j} = h_{i,j} - h_{1,1} + t + i + j - 2. \]
Note that $F$ is well-defined, since if $dh = dh'$, then $h - h'$ is constant, and $h_{i,j} - h_{1,1} = h'_{i,j} - h'_{1,1}$ for all $i,j$. To see that $F$ does indeed map $\mathcal{S}_1 \times \mathbb{F}_3$ into $\mathcal{S}_2$, observe that
\begin{align*}
    (D_x F(dh,t))dx + (D_y F(dh,t))dy &= dF(dh,t)\\
    &= dh + dx + dy.
\end{align*}
If we write $dh= fdx + gdy$, then $f$ and $g$ are nowhere equal to 2, since $dh$ is admissible. So, comparing the first and last expressions in the equation above, we see that $D_xF(dh,t)$ and $D_yF(dh,t)$ are nowhere 0. Thus, $F$ is a well-defined map into $\mathcal{S}_2$. Next, we define $G: \mathcal{S}_2 \to \mathcal{S}_1 \times \mathbb{F}_3$ by
\[G(h) = (dh - dx - dy, h_{1,1}). \]
To see that $G$ maps $\mathcal{S}_2$ into $\mathcal{S}_1 \times \mathbb{F}_3$, observe that if we write $dh = fdx + gdy$, then $f$ and $g$ are nowhere zero, so the components of $dh - dx - dy$ are nowhere equal to 2. Now let $x,y:[m+1]\times[n+1]\to\mathbb{F}_3$ be functions defined by
\begin{align*}
    x_{i,j} &= i\\
    y_{i,j} &= j.
\end{align*}
Then we see that
\begin{align*}
    (F \circ G)(h)_{i,j} &= F(d(h - x - y),h_{1,1})_{i,j}\\
    &= (h_{i,j} - x_{i,j} - y_{i,j}) - (h_{1,1} - x_{1,1} - y_{1,1}) + h_{1,1} +  i + j - 2 \\
    &= h_{i,j}
\end{align*}
and
\begin{align*}
    (G \circ F)(dh,t) &= (dF(dh,t) - dx - dy, F(dh,t)_{1,1}) 
    \\&= ((dh + dx + dy) - dx - dy, t)\\
    &= (dh, t).
\end{align*}
Thus, $F$ and $G$ are inverses of each other, and $\mathcal{S}_1 \times \mathbb{F}_3$ and $\mathcal{S}_2$ are in bijection.
\end{proof}
\section{Toroidal Boundary Conditions}
We will now apply the methods used in the previous section to lattices with toroidal boundary conditions. In this section, we will make the additional assumption that $m$ and $n$ are not divisible by 3. Let $A$ be an admissible state of the six-vertex model, with vertical entries $f_{i,j}$ and horizontal entries $g_{i,j}$ as previously. We say that $A$ has \emph{toroidal boundary conditions} if its boundary values satisfy
\begin{align*}
    g_{1,j} &= g_{m+1,j}\\
    f_{i,1} &= f_{i,n+1}
\end{align*}
for all $1 \leq i \leq m$ and $1 \leq j \leq n$.
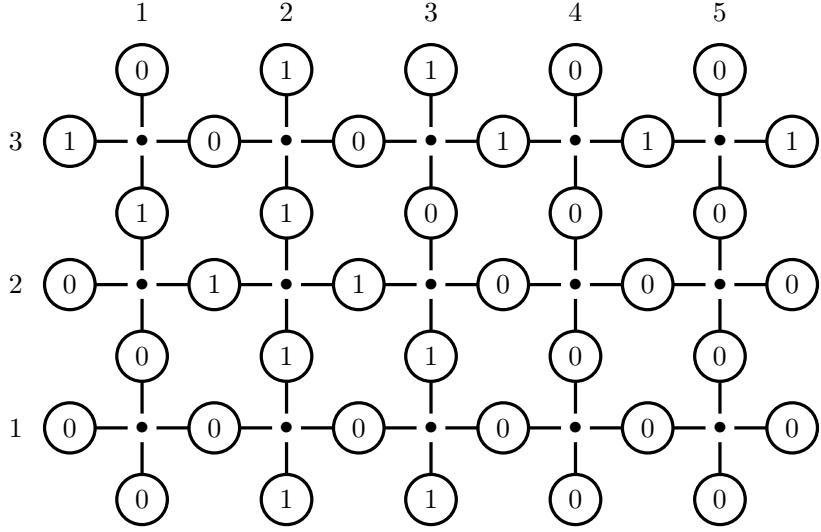
\begin{figure}[h]
\[
  \scalebox{.95}{\begin{tikzpicture}
    \draw [line width=0.45mm] (1,0)--(1,6);
    \draw [line width=0.45mm] (3,0)--(3,6);
    \draw [line width=0.45mm] (5,0)--(5,6);
    \draw [line width=0.45mm] (7,0)--(7,6);
    \draw [line width=0.45mm] (9,0)--(9,6);
    \draw [line width=0.45mm] (0,1)--(10,1);
    \draw [line width=0.45mm] (0,3)--(10,3);
    \draw [line width=0.45mm] (0,5)--(10,5);
    \draw[line width=0.45mm, fill=white] (1,6) circle (.35);
    \draw[line width=0.45mm, fill=white] (3,6) circle (.35);
    \draw[line width=0.45mm, fill=white] (5,6) circle (.35);
    \draw[line width=0.45mm, fill=white] (7,6) circle (.35);
    \draw[line width=0.45mm, fill=white] (9,6) circle (.35);
    \draw[line width=0.45mm, fill=white] (1,4) circle (.35);
    \draw[line width=0.45mm, fill=white] (3,4) circle (.35);
    \draw[line width=0.45mm, fill=white] (5,4) circle (.35);
    \draw[line width=0.45mm, fill=white] (7,4) circle (.35);
    \draw[line width=0.45mm, fill=white] (9,4) circle (.35);
    \draw[line width=0.45mm, fill=white] (1,2) circle (.35);
    \draw[line width=0.45mm, fill=white] (3,2) circle (.35);
    \draw[line width=0.45mm, fill=white] (5,2) circle (.35);
    \draw[line width=0.45mm, fill=white] (7,2) circle (.35);
    \draw[line width=0.45mm, fill=white] (9,2) circle (.35);
    \draw[line width=0.45mm, fill=white] (1,0) circle (.35);
    \draw[line width=0.45mm, fill=white] (3,0) circle (.35);
    \draw[line width=0.45mm, fill=white] (5,0) circle (.35);
    \draw[line width=0.45mm, fill=white] (7,0) circle (.35);
    \draw[line width=0.45mm, fill=white] (9,0) circle (.35);
    \draw[line width=0.45mm, fill=white] (0,5) circle (.35);
    \draw[line width=0.45mm, fill=white] (2,5) circle (.35);
    \draw[line width=0.45mm, fill=white] (4,5) circle (.35);
    \draw[line width=0.45mm, fill=white] (6,5) circle (.35);
    \draw[line width=0.45mm, fill=white] (8,5) circle (.35);
    \draw[line width=0.45mm, fill=white] (10,5) circle (.35);
    \draw[line width=0.45mm, fill=white] (0,3) circle (.35);
    \draw[line width=0.45mm, fill=white] (2,3) circle (.35);
    \draw[line width=0.45mm, fill=white] (4,3) circle (.35);
    \draw[line width=0.45mm, fill=white] (6,3) circle (.35);
    \draw[line width=0.45mm, fill=white] (8,3) circle (.35);
    \draw[line width=0.45mm, fill=white] (10,3) circle (.35);
    \draw[line width=0.45mm, fill=white] (0,1) circle (.35);
    \draw[line width=0.45mm, fill=white] (2,1) circle (.35);
    \draw[line width=0.45mm, fill=white] (4,1) circle (.35);
    \draw[line width=0.45mm, fill=white] (6,1) circle (.35);
    \draw[line width=0.45mm, fill=white] (8,1) circle (.35);
    \draw[line width=0.45mm, fill=white] (10,1) circle (.35);

    \path[fill=white] (1,1) circle (.2);
    \node at (1,1) {$\bullet$};
    \path[fill=white] (3,1) circle (.2);
    \node at (3,1) {$\bullet$};
    \path[fill=white] (5,1) circle (.2);
    \node at (5,1) {$\bullet$};
    \path[fill=white] (7,1) circle (.2);
    \node at (7,1) {$\bullet$};
    \path[fill=white] (9,1) circle (.2);
    \node at (9,1) {$\bullet$};

    \path[fill=white] (1,3) circle (.2);
    \node at (1,3) {$\bullet$};
    \path[fill=white] (3,3) circle (.2);
    \node at (3,3) {$\bullet$};
    \path[fill=white] (5,3) circle (.2);
    \node at (5,3) {$\bullet$};
    \path[fill=white] (7,3) circle (.2);
    \node at (7,3) {$\bullet$};
    \path[fill=white] (9,3) circle (.2);
    \node at (9,3) {$\bullet$};

    \path[fill=white] (1,5) circle (.2);
    \node at (1,5) {$\bullet$};
    \path[fill=white] (3,5) circle (.2);
    \node at (3,5) {$\bullet$};
    \path[fill=white] (5,5) circle (.2);
    \node at (5,5) {$\bullet$};
    \path[fill=white] (7,5) circle (.2);
    \node at (7,5) {$\bullet$};
    \path[fill=white] (9,5) circle (.2);
    \node at (9,5) {$\bullet$};

    \node at (1,6) {$0$};
    \node at (3,6) {$1$};
    \node at (5,6) {$1$};
    \node at (7,6) {$0$};
    \node at (9,6) {$0$};
    \node at (1,4) {$1$};
    \node at (3,4) {$1$};
    \node at (5,4) {$0$};
    \node at (7,4) {$0$};
    \node at (9,4) {$0$};
    \node at (1,2) {$0$};
    \node at (3,2) {$1$};
    \node at (5,2) {$1$};
    \node at (7,2) {$0$};
    \node at (9,2) {$0$};
    \node at (1,0) {$0$};
    \node at (3,0) {$1$};
    \node at (5,0) {$1$};
    \node at (7,0) {$0$};
    \node at (9,0) {$0$};
    \node at (0,5) {$1$};
    \node at (2,5) {$0$};
    \node at (4,5) {$0$};
    \node at (6,5) {$1$};
    \node at (8,5) {$1$};
    \node at (10,5) {$1$};
    \node at (0,3) {$0$};
    \node at (2,3) {$1$};
    \node at (4,3) {$1$};
    \node at (6,3) {$0$};
    \node at (8,3) {$0$};
    \node at (10,3) {$0$};
    \node at (0,1) {$0$};
    \node at (2,1) {$0$};
    \node at (4,1) {$0$};
    \node at (6,1) {$0$};
    \node at (8,1) {$0$};
    \node at (10,1) {$0$};
    \node at (1.00,6.8) {$ 1$};
    \node at (3.00,6.8) {$ 2$};
    \node at (5.00,6.8) {$ 3$};
    \node at (7.00,6.8) {$ 4$};
    \node at (9.00,6.8) {$ 5$};

    \node at (-.75,1) {$ 1$};
    \node at (-.75,3) {$ 2$};
    \node at (-.75,5) {$ 3$};
  \end{tikzpicture}}
\]
\caption{An admissible state with toroidal boundary conditions.}
\end{figure}
If $h: \mathbb{Z} \times \mathbb{Z} \to \mathbb{F}_3$ is a function, we say that $h$ is \emph{doubly periodic} if
\begin{align*}
    h_{i+m,j} = h_{i,j + n} =  h_{i,j}
\end{align*}
for all $i,j \in \mathbb{Z}$. Just as before, for such a function $h$, we define its \emph{partial derivatives} $D_xh, D_yh: \mathbb{Z} \times \mathbb{Z} \to \mathbb{F}_3$ by
\[(D_xh)_{i,j} = h_{i+1,j} - h_{i,j} \]
and
\[(D_yh)_{i,j} = h_{i,j+1}- h_{i,j}. \]
Note that $D_xh$ and $D_yh$ are doubly periodic. We define a \emph{toroidal 1-form} to be a formal expression $fdx + gdy$, where $f,g: \mathbb{Z} \times \mathbb{Z} \to \mathbb{F}_3$ are doubly periodic functions. For a doubly periodic function $h$, we define its \emph{exterior derivative} $dh$ as the toroidal 1-form given by
\[dh = (D_xh)dx + (D_yh)dy. \]
We say that a toroidal 1-form $fdx + gdy$ is \emph{closed} if
\[D_yf = D_xg\]
and \emph{exact} if $fdx + gdy = dh$ for some doubly periodic $h: \mathbb{Z} \times \mathbb{Z} \to \mathbb{F}_3$. A toroidal 1-form $fdx + gdy$ is \emph{admissible} if for all $i$ and $j$, $f_{i,j}$ and $g_{i,j}$ are not equal to 2. In this section, we will simply refer to a toroidal 1-form as a 1-form. The following lemma is an analogue of Lemma \ref{poincare}; it computes the 1-dimensional cohomology of the discrete torus.
\begin{lemma}\label{cohomology}
Every closed 1-form can be written uniquely in the form
\[rdx + sdy + \omega, \]
where $r,s \in \mathbb{F}_3$ and $\omega$ is exact.
\end{lemma}
\begin{proof}
Let $fdx + gdy$ be a closed 1-form, where $f,g:\mathbb{Z}\times\mathbb{Z}\to\mathbb{F}_3$ are doubly periodic. Let
\begin{align*}r &= \frac{1}{m}\sum_{i = 1}^{m}f_{i,1},\\ s &= \frac{1}{n}\sum_{j=1}^{n} g_{1,j}.
\end{align*}
We claim that $\omega = fdx + gdy - rdx - sdy$ is exact. To see this, define a function $h: \mathbb{Z} \times \mathbb{Z} \to \mathbb{F}_3$ by
\[h_{i,j} = \sum_{a = 1}^{\tilde{i}-1}(f_{a,1} - r) + \sum_{b = 1}^{\tilde{j} - 1}(g_{i,b} - s),\]
where $\tilde{i}$ is the unique integer such that $1 \leq \tilde{i} \leq m$ with $\tilde{i} \equiv i \pmod{m}$, and $\tilde{j}$ is the integer such that $1 \leq \tilde{j} \leq n$ with $\tilde{j} \equiv j \pmod{n}$. It is clear that $h$ is doubly periodic. Observe that if $i \not\equiv 0 \pmod{m}$, we have $\widetilde{i + 1} = \tilde{i} + 1$, so
\[ \sum_{a = 1}^{\widetilde{i + 1}-1}(f_{a,1} - r) = \sum_{a = 1}^{\tilde{i}}(f_{a,1} - r).\]
If $i \equiv 0 \pmod{m}$, then we also have
\begin{align*}
    \sum_{a = 1}^{\widetilde{i + 1}-1}(f_{a,1} - r) &= \sum_{a = 1}^{0}(f_{a,1} - r) \\
    &= 0 \\
    &= \sum_{a = 1}^{m} (f_{a,1} - r)\\
    &= \sum_{a = 1}^{\tilde{i}}(f_{a,1} - r).
\end{align*}
A similar argument shows that for all $j$,
\[\sum_{b = 1}^{\widetilde{j + 1} - 1}(g_{i,b} - s) = \sum_{b = 1}^{\tilde{j}}(g_{i,b} - s).\]
We now compute
\begin{align*}
    (D_x h)_{i,j} &= \left(\sum_{a = 1}^{\widetilde{i+1}-1}(f_{a,1} - r) + \sum_{b = 1}^{\tilde{j} - 1}(g_{i+1,b} - s)\right) - \left(\sum_{a = 1}^{\tilde{i}-1}(f_{a,1} - r) + \sum_{b = 1}^{\tilde{j} - 1}(g_{i,b} - s)\right) \\
    &= \left(\sum_{a = 1}^{\tilde{i}}(f_{a,1} - r) + \sum_{b = 1}^{\tilde{j} - 1}(g_{i+1,b} - s)\right) - \left(\sum_{a = 1}^{\tilde{i}-1}(f_{a,1} - r) + \sum_{b = 1}^{\tilde{j} - 1}(g_{i,b} - s)\right)\\
    &= (f_{\tilde{i},1} - r) + \sum_{b = 1}^{\tilde{j}-1}(g_{i+1,b} - g_{i,b}) \\
    &= (f_{i,1} - r) + \sum_{b = 1}^{\tilde{j}-1}(D_xg)_{i,b}\\
    &= (f_{i,1} - r) + \sum_{b = 1}^{\tilde{j}-1}(D_yf)_{i,b}\\
    &= f_{i,\tilde{j}} - r \\
    &= f_{i,j} - r
\end{align*}
and
\begin{align*}
    (D_yh)_{i,j} &= \left(\sum_{a = 1}^{\tilde{i}-1}(f_{a,1} - r) + \sum_{b = 1}^{\widetilde{j+1} - 1}(g_{i,b} - s)\right) - \left(\sum_{a = 1}^{\tilde{i}-1}(f_{a,1} - r) + \sum_{b = 1}^{\tilde{j} - 1}(g_{i,b} - s) \right) \\
    &= \left(\sum_{a = 1}^{\tilde{i}-1}(f_{a,1} - r) + \sum_{b = 1}^{\tilde{j}}(g_{i,b} - s)\right) - \left(\sum_{a = 1}^{\tilde{i}-1}(f_{a,1} - r) + \sum_{b = 1}^{\tilde{j} - 1}(g_{i,b} - s) \right)\\
    &= g_{i,\tilde{j}} - s \\
    &= g_{i,j} - s.
\end{align*}
Thus, we have shown that $dh = \omega$, so $\omega$ is exact.

To check uniqueness, it suffices to show that if $rdx + sdy$ is exact, then $r = s = 0$. Suppose that $rdx + sdy = df$, where $f$ is doubly periodic. Then
\begin{align*}
    0 &= \frac{1}{m} \sum_{i = 1}^{m}(f_{i+1,1} - f_{i,1})\\
    &= \frac{1}{m}\sum_{i = 1}^{m}(D_xf)_{i,1}\\
    &= \frac{1}{m}\sum_{i = 1}^{m}r \\
    &= r.
\end{align*}
A similar argument shows that $s = 0$.
\end{proof}
\begin{lemma}
There is a one-to-one correspondence between admissible closed 1-forms and admissible states of the six-vertex model with toroidal boundary conditions.
\end{lemma}
\begin{proof}
Let $f_{i,j}$ and $g_{i,j}$ be the vertical and horizontal entries of an admissible state $A$ with toroidal boundary conditions. Then we may uniquely extend $f_{i,j}$ and $g_{i,j}$ to doubly periodic functions $\tilde{f},\tilde{g}:\mathbb{Z}\times\mathbb{Z}\to\mathbb{F}_3$. We associate $A$ with the closed admissible 1-form $\tilde{f}dx + \tilde{g}dy$. It is not difficult to see that this is a one-to-one correspondence.
\end{proof}
Call a doubly periodic function $h: \mathbb{Z} \times \mathbb{Z} \to \mathbb{F}_3$ \emph{sparse} if neither $D_xh$ nor $D_yh$ is surjective, and $h_{1,1} = 0$. We can characterize admissible states with toroidal boundary conditions in terms of sparse functions, in the same way that we classified admissible states in the last section in terms of 3-colorings.
\begin{theorem}
There is a one-to-one correspondence between sparse functions and admissible states of the six-vertex model with toroidal boundary conditions.
\end{theorem}
\begin{proof}
Let $\mathcal{S}_1$ denote the set of closed admissible 1-forms, and let $\mathcal{S}_2$ be the set of sparse functions. It suffices to find a bijection between $\mathcal{S}_1$ and $\mathcal{S}_2$. Define a map $F: \mathcal{S}_2 \to \mathcal{S}_1$ as follows. If $h$ is a sparse function, define
\[F(h) = r(h)dx + s(h)dy + (D_xh)dx + (D_yh)dy = r(h)dx + s(h)dy + dh, \]
where $2-r(h) \notin \text{Im}(D_xh)$ and $2-s(h) \notin \text{Im}(D_yh)$. We will show that $F$ is a bijection. To see that $F$ is injective, suppose that $F(h) = F(h')$. Then, by the uniqueness result in Lemma \ref{cohomology}, $dh = dh'$, which implies that $h - h'$ is constant. But $h_{1,1} = (h')_{1,1} = 0$, so $h = h'$ and $F$ is injective. To see that $F$ is surjective, let
\[\omega = rdx + sdy + (D_xh)dx + (D_yh)dy \]
be an element of $\mathcal{S}_1$ (by Lemma \ref{cohomology}). Without loss of generality, we may assume that $h_{1,1} = 0$. Then, by definition, $r + D_xh$ and $s + D_yh$ are not surjective, implying that $D_xh$ and $D_yh$ are not surjective. So $\omega = F(h)$. Thus, we have a bijection.
\end{proof}
\section{Admissible States of the Eight-Vertex Model}
In this section, we will consider a different set of admissible labelings.
\begin{figure}[H]
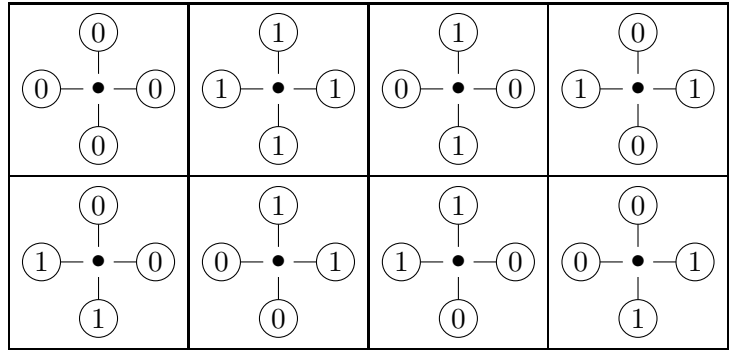

\[
\begin{array}{|c|c|c|c|c|c|}
\hline
  \gammaice{0}{0}{0}{0} &
  \gammaice{1}{1}{1}{1} &
  \gammaice{0}{1}{0}{1} &
  \gammaice{1}{0}{1}{0}\\
\hline
  \gammaice{1}{0}{0}{1} &
  \gammaice{0}{1}{1}{0} &
  \gammaice{1}{1}{0}{0} &
  \gammaice{0}{0}{1}{1}
  \\
\hline\end{array}\]
\caption{Admissible labelings in the eight-vertex model.}
\label{uncoloredbw}
\end{figure}
In the eight-vertex model, we call a state \emph{admissible} if all of its vertices are in one of the eight configurations of Figure 4. Let $f_{i,j}$ and $g_{i,j}$ be the labelings of the vertical edges and horizontal edges as before. At a vertex $v_{i,j}$, we have the labeling
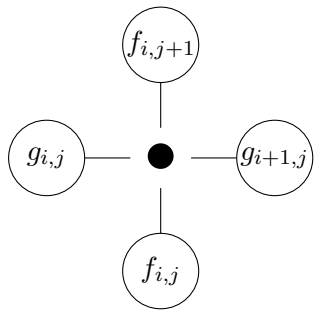
\begin{figure}[H]
\[\begin{tikzpicture}[scale=2]
\coordinate (a) at (-.75, 0);
\coordinate (b) at (0, .75);
\coordinate (c) at (.75, 0);
\coordinate (d) at (0, -.75);
\coordinate (aa) at (-.75,.5);
\coordinate (cc) at (.75,.5);
\draw (a)--(c);
\draw (b)--(d);
\draw[fill=white] (a) circle (.25);
\draw[fill=white] (b) circle (.25);
\draw[fill=white] (c) circle (.25);
\draw[fill=white] (d) circle (.25);
\node at (0,1) { };
\node at (a) {$g_{i,j}$};
\node at (b) {$f_{i,j+1}$};
\node at (c) {$g_{i+1,j}$};
\node at (d) {$f_{i,j}$};
\path[fill=white] (0,0) circle (.2);
\node at (0,0) {\Huge $\bullet$};
\end{tikzpicture}\]
\end{figure}
as before, so a state is admissible exactly when
\begin{align}
f_{i,j} + g_{i,j} + f_{i,j+1} + g_{i+1,j} \equiv 0 \pmod 2 \label{admissible8}
\end{align}
for all $1 \leq i \leq m$ and $1 \leq j \leq n$. So in this section, we will view $f$ and $g$ as functions from $[m] \times [n+1]$ to $\mathbb{F}_3$ and $[m+1] \times [n]$ to $\mathbb{F}_3$ respectively. Since the condition for a state to be admissible is a linear condition in this case, it is significantly easier to find the number of admissible states.
\begin{theorem}\label{8admissible}
The number of admissible states of the eight-vertex model is $2^{m + n + mn}$.
\end{theorem}
\begin{proof}
Let $V$ be the $\mathbb{F}_2$-vector space of pairs of functions $(f,g)$ with $f: [m] \times [n+1] \to \mathbb{F}_2$ and $g: [m+1] \times [n] \to \mathbb{F}_2$. Let $W$ be $mn$-dimensional vector space with basis vectors $e_{i,j}$ for $1\leq i \leq m$ and $1 \leq j \leq n$. Define a linear map $\varphi: V \to W$ by
\[\varphi(f,g) = \sum_{i = 1}^m \sum_{j = 1}^n (f_{i,j} + g_{i,j} + f_{i,j+1} + g_{i+1,j})e_{i,j}. \]
By equation (\ref{admissible8}), we may view $V' = \ker \varphi$ as the set of admissible states of the eight-vertex model. Fix some $a \in [m]$ and $b \in [n]$. Let $g = 0$ and define $f$ by $f_{i,j} = 0$ if  $i \neq a$ or $j \leq b$, and $f_{i,j} = 1$ otherwise. Then,
\begin{align*}
    \varphi(f,g) &= \sum_{i = 1}^m \sum_{j = 1}^n (f_{i,j} + f_{i,j + 1})e_{i,j} \\
    &= \sum_{j = b}^n (f_{a,j} + f_{a,j + 1})e_{a,j}\\
    &= e_{a,b}.
\end{align*}
Since $a$ and $b$ were arbitrary, this shows that $\varphi$ is surjective. Thus,
\begin{align*}
    \dim V' = \dim V - \dim W = m(n + 1) + (m + 1)n - mn = mn + m + n.
\end{align*}
So $V'$ contains $2^{m + n + mn}$ elements, and the result follows.
\end{proof}
Now we focus our attention on the number of admissible states with a given set of boundary conditions.
\begin{theorem}\label{sumto0}
There exists an admissible state $(f,g)$ with boundary conditions
\[f_{i,1}, f_{i,n+1},g_{1,j},g_{m+1,j} \]
if and only if
\begin{align}\label{goodBC}
    \sum_{i = 1}^m(f_{i,1} + f_{i,n+1}) + \sum_{j = 1}^n(g_{1,j} + g_{m + 1,j}) = 0.
\end{align}
\end{theorem}
\begin{proof}
Consider an admissible state $(f,g)$ with vertical edges $f_{i,j}$ and horizontal edges $g_{i,j}$. By equation (\ref{admissible8}), we have
\begin{align*}
    0 &= \sum_{i = 1}^m \sum_{j = 1}^n (f_{i,j} + g_{i,j} + f_{i,j+1} + g_{i+1,j}) \\
    &= \sum_{i = 1}^m\sum_{j = 1}^n (f_{i,j} + g_{i,j}) + \sum_{i = 1}^m\sum_{j = 2}^{n + 1} f_{i,j} + \sum_{i = 2}^{m+1}\sum_{j = 1}^n g_{i,j}\\
    &= \sum_{i = 1}^m \left(f_{i,1} + \sum_{j = 2}^nf_{i,j} + \sum_{j=1}^ng_{i,j} + f_{i,n+1} + \sum_{j = 2}^n f_{i,j}\right) + \sum_{i = 2}^{m+1}\sum_{j = 1}^n g_{i,j} \\
    &= \sum_{i = 1}^m(f_{i,1} + f_{i,n+1}) + \sum_{i = 1}^m\sum_{j = 1}^n g_{i,j} + \sum_{i = 2}^{m+1}\sum_{j=1}^ng_{i,j}\\
    &= \sum_{i = 1}^m(f_{i,1} + f_{i,n+1}) + \sum_{j = 1}^n(g_{1,j}+g_{m+1,j}).
\end{align*}
To prove the converse, let \[f_{i,1}, f_{i,n+1},g_{1,j},g_{m+1,j}\]
be prescribed boundary conditions such that equation (\ref{goodBC}) holds. We will find an admissible state $(f,g)$ with these boundary conditions. Let us define
\[g_{i,j} = 0\]
for $2 \leq i \leq m$ and $2 \leq j \leq n$, and
\[g_{i,1} = \sum_{b = 1}^n g_{1,b} + \sum_{a = 1}^{i-1}(f_{a,1}+f_{a,n+1})\]
for $2 \leq i \leq m$. We also define
\[f_{i,j} = f_{i,n+1} \]
for $2 \leq i \leq m - 1$ and $2 \leq j \leq n$,
\[f_{1,j} = f_{1,n+1} + \sum_{b = 1}^{n - j + 1}g_{1,n - b + 1}\]
for $2 \leq j \leq n$, and
\[f_{m,j} = f_{m,n+1} + \sum_{b = 1}^{n - j + 1}g_{m+1,n - b + 1} \]
for $2 \leq j \leq n$. We claim that this defines an admissible state. To see this, we will look at several different cases. If $i = 1$ and $j = 1$, then
\begin{align*}
    f_{i,j} + f_{i,j+1} + g_{i,j} + g_{i + 1,j} &= f_{1,1} + \left(f_{1,n+1} + \sum_{b = 1}^{n-1} g_{1,n-b+1}\right)\\&+ g_{1,1} + \left(\sum_{b = 1}^n g_{1,b} + f_{1,1} + f_{1,n+1} \right) \\
    &= 0.
\end{align*}
If $i = 1$ and $2 \leq j \leq n - 1$, then
\begin{align*}
    f_{i,j} + f_{i,j+1} + g_{i,j} + g_{i + 1,j} &= \left(f_{1,n+1} + \sum_{b = 1}^{n - j + 1}g_{1,n-b+1}\right) + \left(f_{1,n+1} + \sum_{b = 1}^{n - j} g_{1,n - b + 1} \right)+ g_{1,j} + 0\\
    &= 0.
\end{align*}
If $i = 1$ and $j = n$, then
\begin{align*}
     f_{i,j} + f_{i,j+1} + g_{i,j} + g_{i + 1,j} &= \left(f_{1,n+1} + g_{1,n} \right) + f_{1,n+1} + g_{1,n} + 0\\
     &= 0.
\end{align*}
If $2 \leq i \leq m - 1$ and $j = 1$, then
\begin{align*}
     f_{i,j} + f_{i,j+1} + g_{i,j} + g_{i + 1,j} &= f_{i,1} + f_{i,n+1} + \left(\sum_{b = 1}^n g_{1,b} + \sum_{a = 1}^{i - 1}(f_{a,1} + f_{a,n+1}) \right) \\&+ \left(\sum_{b = 1}^n g_{1,b}
     + \sum_{a = 1}^i (f_{a,1} + f_{a,n+1})\right)\\
     &= 0.
\end{align*}
If $2 \leq i \leq m - 1$ and $2 \leq j \leq n - 1$, then
\begin{align*}
     f_{i,j} + f_{i,j+1} + g_{i,j} + g_{i + 1,j} &= f_{i,n+1} + f_{i,n+1} + 0 + 0\\
     &= 0.
\end{align*}
If $2 \leq i \leq m - 1$ and $j = n$, then
\begin{align*}
     f_{i,j} + f_{i,j+1} + g_{i,j} + g_{i + 1,j} &= f_{i,n+1} + f_{i,n+1} + 0 + 0\\
     &= 0.
\end{align*}
If $i = m$ and $j = 1$, then by equation (\ref{goodBC}),
\begin{align*}
     f_{i,j} + f_{i,j+1} + g_{i,j} + g_{i + 1,j} &= f_{m,1} + \left(f_{m,n+1} + \sum_{b = 1}^{n - 1}g_{m+1,n-b+1} \right)\\
     &+ \left(\sum_{b = 1}^n g_{1,b} + \sum_{a = 1}^{m - 1}(f_{a,1} + f_{a,n+1}) \right) + g_{m+1,1} \\
     &= \sum_{b = 1}^{n}(g_{1,b} + g_{m + 1,b}) + \sum_{a = 1}^m (f_{a,1} + f_{a,n+1}) \\
     &= 0.
\end{align*}
If $i = m$ and $2 \leq j \leq n - 1$, then
\begin{align*}
    f_{i,j} + f_{i,j+1} + g_{i,j} + g_{i + 1,j} &= \left(f_{m,n+1} + \sum_{b = 1}^{n - j + 1}g_{m + 1,n-b+1} \right) + \left(f_{m,n+1} + \sum_{b = 1}^{n - j}g_{m+1,n-b+1} \right)\\
    &+0 + g_{m+1,j}\\
    &= 0.
\end{align*}
If $i = m$ and $j = n$, then
\begin{align*}
    f_{i,j} + f_{i,j+1} + g_{i,j} + g_{i + 1,j} &= \left(f_{m,n+1} + g_{m + 1,n} \right) + f_{m,n+1} + 0 + g_{m + 1,n}\\
    &= 0.
\end{align*}
Thus, the state we have defined is admissible.
\end{proof}
\begin{theorem}
Let $f_{i,1}, f_{i,n+1},g_{1,j},g_{m+1,j}\in \mathbb{F}_2$ be values such that
\[\sum_{i = 1}^m(f_{i,1} + f_{i,n+1}) + \sum_{j = 1}^n(g_{1,j} + g_{m + 1,j}) = 0.\] Then the number of admissible states with boundary conditions $f_{i,1}, f_{i,n+1}, g_{1,j},g_{m+1,j}$ is $2^{(m-1)(n-1)}$.
\end{theorem}
\begin{proof}
Let $\mathcal{S}$ denote the set of admissible states with the boundary conditions above, and let $\mathcal{S}_0$ denote the set of admissible states with boundary values all 0. We claim that $\mathcal{S}$ and $\mathcal{S}_0$ have the same number of elements. To see this, note first that $\mathcal{S}$ is non-empty by Theorem \ref{sumto0}, so let $(f^1, g^1)\in\mathcal{S}$. Then we define a map $\psi: \mathcal{S}_0 \to \mathcal{S}$ by
\[\psi(f^0, g^0) = (f^0 + f^1, g^0 + g^1). \]
It is easy to see that $\psi$ is injective. To see that $\psi$ is surjective, suppose that $(f^2, g^2) \in \mathcal{S}$. Then $(f^2 - f^1, g^2 - g^1) \in \mathcal{S}_0$, so $\psi(f^2 - f^1, g^2 - g^1) = (f^2, g^2)$. Thus, $\psi$ is a bijection between $\mathcal{S}_0$ and $\mathcal{S}$. This shows that all sets of boundary conditions having at least one admissible state have the same number of admissible states. It is not difficult to see that the number of sets of boundary conditions satisfying (\ref{goodBC}) is $2^{2m + 2n - 1}$. But by Theorem \ref{8admissible}, there are $2^{m + n + mn}$ admissible states across all boundary conditions. So for a given set of boundary conditions, there are
\[\frac{2^{m + n + mn}}{2^{2m + 2n - 1}} = 2^{(m - 1)(n - 1)} \]
admissible states.
\end{proof}
\section{Yang-Baxter Equation for the Eight-Vertex Model}
We continue to study the eight-vertex model. First, we review some basic properties of the Yang-Baxter equation. To each admissible labeling, we assign a value in some field $\mathbb{F}$, which we call its \emph{Boltzmann weight}.
\begin{figure}[H]
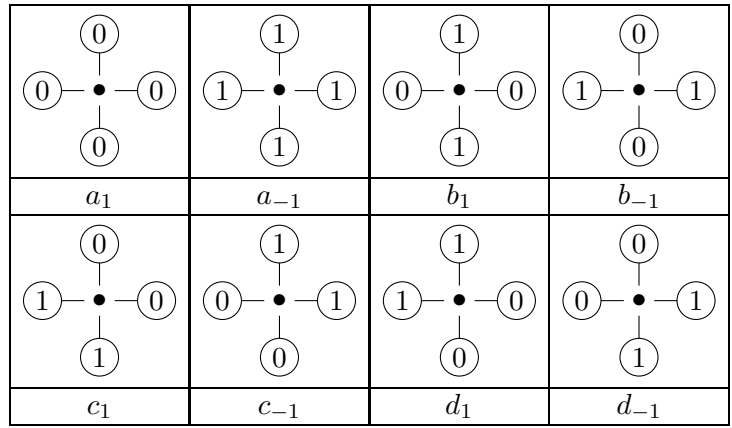

\[
\begin{array}{|c|c|c|c|c|c|}
\hline
  \gammaice{0}{0}{0}{0} &
  \gammaice{1}{1}{1}{1} &
  \gammaice{0}{1}{0}{1} &
  \gammaice{1}{0}{1}{0}\\
\hline
a_1 & a_{-1} & b_1 & b_{-1}\\
\hline
  \gammaice{1}{0}{0}{1} &
  \gammaice{0}{1}{1}{0} &
  \gammaice{1}{1}{0}{0} &
  \gammaice{0}{0}{1}{1}
  \\
\hline
c_1 & c_{-1} & d_1 & d_{-1}\\
\hline\end{array}\]
\caption{Admissible labelings, along with their Boltzmann weights.}
\label{uncoloredbw}
\end{figure}
We will also assign Boltzmann weights to diagonally oriented vertices.
\begin{figure}[H]
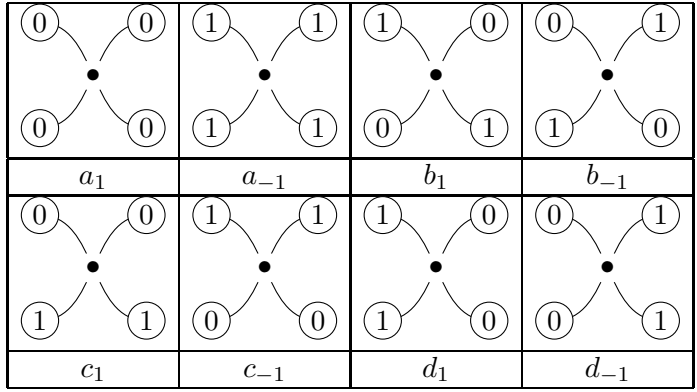

\[
\begin{array}{|c|c|c|c|c|c|}
\hline
  \gammagamma{0}{0}{0}{0} &
  \gammagamma{1}{1}{1}{1} &
  \gammagamma{0}{1}{0}{1} &
  \gammagamma{1}{0}{1}{0}\\
\hline
a_1 & a_{-1} & b_1 & b_{-1}\\
\hline
  \gammagamma{1}{0}{0}{1} &
  \gammagamma{0}{1}{1}{0} &
  \gammagamma{1}{1}{0}{0} &
  \gammagamma{0}{0}{1}{1}
  \\
\hline
c_1 & c_{-1} & d_1 & d_{-1}\\
\hline\end{array}\]
\caption{Admissible labelings, along with their Boltzmann weights, for diagonally oriented vertices.}
\label{uncoloredbw}
\end{figure}
To each such assignment of Boltzmann weights, we associate a matrix
\[R = \begin{pmatrix} a_1 & & & d_1\\
& b_1 & c_1 & \\
& c_{-1} & b_{-1} &\\
d_{-1} & & & a_{-1}
\end{pmatrix} = \begin{pmatrix} a_1(R) & & & d_1(R)\\
& b_1(R) & c_1(R) & \\
& c_{-1}(R) & b_{-1}(R) &\\
d_{-1}(R) & & & a_{-1}(R)
\end{pmatrix}.\]
Let $V$ be a vector space with basis $v_0$ and $v_1$. Then we may view $R$ as an endomorphism of $V \otimes V$ with respect to the basis $v_0 \otimes v_0$, $v_0 \otimes v_1$, $v_1 \otimes v_0$, and $v_1 \otimes v_1$. We write
\[R(v_\nu \otimes v_\beta) = \sum_{\theta,\gamma}R_{\nu\beta}^{\theta\gamma}v_\theta \otimes v_\gamma.\]
For example, we have $R_{01}^{10} = c_{-1}(R)$. Observe that $R_{\nu\beta}^{\theta\gamma}$ is the Boltzmann weight of the following labelings.
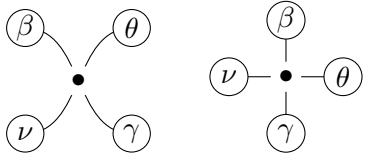
\begin{figure}[H]
\[
\gammagamma{\nu}{\beta}{\theta}{\gamma} \qquad \gammaice{\nu}{\beta}{\theta}{\gamma}\]
\label{uncoloredbw}
\end{figure}

Now if $\phi$ is an endormorphism of $V \otimes V$, we define endomorphisms $\phi_{12}, \phi_{23}, \phi_{13}$ of $V\otimes V \otimes V$ as follows. If $\phi = \phi' \otimes \phi''$ for $\phi', \phi'' \in \text{End}(V)$, then we define
\begin{align*}
    \phi_{12} &= \phi' \otimes \phi'' \otimes 1\\
    \phi_{13} &= \phi' \otimes 1 \otimes \phi''\\
    \phi_{23} &= 1 \otimes \phi \otimes \phi''.
\end{align*}
We extend these definitions for all $\phi$ by linearity. For $\phi, \psi,\chi \in \text{End}(V\otimes V \otimes V)$, we define their \emph{Yang-Baxter commutator} $[[\phi,\psi,\chi]]$ by
\[[[\phi,\psi,\chi]] = \phi_{12}\psi_{13}\chi_{23} - \chi_{23}\psi_{13}\phi_{12}.\]
For $R, S, T \in \text{End}(V \otimes V)$, we say that the \emph{star-triangle relation} holds if
\begin{equation}
\label{eqn:ybe}
\hfill
\sum_{\gamma,\mu,\nu}\quad
\begin{tikzpicture}[baseline=(current bounding box.center)]
  \draw (0,1) to [out = 0, in = 180] (2,3) to (4,3);
  \draw (0,3) to [out = 0, in = 180] (2,1) to (4,1);
  \draw (3,0) to (3,4);
  \draw[fill=white] (0,1) circle (.3);
  \draw[fill=white] (0,3) circle (.3);
  \draw[fill=white] (3,4) circle (.3);
  \draw[fill=white] (4,3) circle (.3);
  \draw[fill=white] (4,1) circle (.3);
  \draw[fill=white] (3,0) circle (.3);
  \draw[fill=white] (2,3) circle (.3);
  \draw[fill=white] (2,1) circle (.3);
  \draw[fill=white] (3,2) circle (.3);
  \node at (0,1) {$\sigma$};
  \node at (0,3) {$\tau$};
  \node at (3,4) {$\beta$};
  \node at (4,3) {$\theta$};
  \node at (4,1) {$\rho$};
  \node at (3,0) {$\alpha$};
  \node at (2,3) {$\nu$};
  \node at (3,2) {$\gamma$};
  \node at (2,1) {$\mu$};
\path[fill=white] (3,3) circle (.3);
\node at (3,3) {$S$};
\path[fill=white] (3,1) circle (.3);
\node at (3,1) {$T$};
\path[fill=white] (1,2) circle (.3);
\node at (1,2) {$R$};
\end{tikzpicture}\quad
= \sum_{\delta,\phi,\psi}\quad
\begin{tikzpicture}[baseline=(current bounding box.center)]
  \draw (0,1) to (2,1) to [out = 0, in = 180] (4,3);
  \draw (0,3) to (2,3) to [out = 0, in = 180] (4,1);
  \draw (1,0) to (1,4);
  \draw[fill=white] (0,1) circle (.3);
  \draw[fill=white] (0,3) circle (.3);
  \draw[fill=white] (1,4) circle (.3);
  \draw[fill=white] (4,3) circle (.3);
  \draw[fill=white] (4,1) circle (.3);
  \draw[fill=white] (1,0) circle (.3);
  \draw[fill=white] (2,3) circle (.3);
  \draw[fill=white] (1,2) circle (.3);
  \draw[fill=white] (2,1) circle (.3);
  \node at (0,1) {$\sigma$};
  \node at (0,3) {$\tau$};
  \node at (1,4) {$\beta$};
  \node at (4,3) {$\theta$};
  \node at (4,1) {$\rho$};
  \node at (1,0) {$\alpha$};
  \node at (2,3) {$\psi$};
  \node at (1,2) {$\delta$};
  \node at (2,1) {$\phi$};
\path[fill=white] (1,3) circle (.3);
\node at (1,3) {$T$};
\path[fill=white] (1,1) circle (.3);
\node at (1,1) {$S$};
\path[fill=white] (3,2) circle (.3);
\node at (3,2) {$R$};
\end{tikzpicture}\quad.
\end{equation}
Here we associate a lattice state with the product of the Boltzmann weights of its vertices. In other words, the star-triangle relation holds when
\[\sum_{\gamma,\mu,\nu} R_{\sigma\tau}^{\nu\mu}S_{\nu\beta}^{\theta\gamma}T_{\mu\gamma}^{\rho\alpha} = \sum_{\delta,\phi,\psi}T_{\tau\beta}^{\psi\delta}S_{\sigma\delta}^{\phi\alpha}R_{\phi\psi}^{\theta\rho}.\]
The main fact we will use in order to do computations is that the star-triangle relation is equivalent to the vanishing of the Yang-Baxter commutator. For a proof of this statement, see \cite{BBF}.
\begin{lemma}
Let $R, S, T \in \End(V\otimes V)$. Then the star-triangle relation holds for $R,S,T$ if and only if $[[R,S,T]] = 0$.
\end{lemma}
Given matrices $S$ and $T$, we will determine necessary conditions for there to exist $R$ such that $[[R,S,T]] = 0$. We will follow the approach of Galleas and Martins \cite{GalleasMartins}. Our analysis will be more general, since we do not assume that the Boltzmann weights $c_1, c_{-1}, d_1, d_{-1}$ satisfy $c_1 = c_{-1}$ and $d_1 = d_{-1}$. Suppose that $R, S, T \in \text{End}(V \otimes V)$ are Boltzmann weights such that $[[R,S,T]] = 0$. Moreover, assume that $a_1(S)$, $b_1(S)$, $c_1(S)$, $d_1(S)$, $a_2(S)$, $b_2(S)$, $c_2(S)$, $d_2(S)$, $a_1(T)$, $b_1(T)$, $c_1(T)$, $d_1(T)$, $a_2(T)$, $b_2(T)$, $c_2(T)$, $d_2(T)$ are nonzero. A computation shows that the condition $[[R,S,T]] = 0$ can be expressed as the system of 28 equations
\begin{align}
    a_j(T)a_j(S)d_i(R) + d_i(T)c_i(S)a_{-j}(R) &= c_i(T)d_i(S)a_j(R) + b_{-j}(T)b_{-j}(S)d_i(R)\label{eqn:1}\\
    d_i(T)b_j(S)c_i(R) + a_j(T)d_i(S)b_{-j}(R) &= b_j(T)d_i(S)a_j(R) + c_{-i}(T)b_{-j}(S)d_i(R)\label{eqn:2}\\
     d_{i}(T)b_j(S)b_j(R) + a_j(T)d_{i}(S)c_{-i}(R) &= d_{i}(T)a_j(S)a_j(R) + a_{-j}(T)c_{-i}(S)d_{i}(R)\label{eqn:3}\\
     c_i(T)a_j(S)c_i(R) + b_j(T)c_i(S)b_{-j}(R) &= a_j(T)c_i(S)a_j(R) + d_{-i}(T)a_{-j}(S)d_{i}(R)\label{eqn:4}\\
     c_i(T)a_j(S)b_j(R) + b_j(T)c_i(S)c_{-i}(R) &= c_i(T)b_j(S)a_j(R) + b_{-j}(T)d_{-i}(S)d_i(R)\label{eqn:5}\\
     b_{-j}(T)a_j(S)c_i(R) + c_{-i}(T)c_i(S)b_{-j}(R) &= d_{-i}(T)d_i(S)b_j(R) + a_j(T)b_{-j}(S)c_i(R)\label{eqn:6}\\
     c_1(T)c_{-1}(S)c_1(R) &= c_{-1}(T)c_1(S)c_{-1}(R)\label{eqn:7}\\
     d_1(T)c_1(S)d_{-1}(R) &= d_{-1}(T)c_{-1}(S)d_1(R)\label{eqn:8}\\
     c_1(T)d_1(S)d_{-1}(R) &= c_{-1}(T)d_{-1}(S)d_1(R)\label{eqn:9}\\
     d_1(T)d_{-1}(S)c_1(R) &= d_{-1}(T)d_1(S)c_{-1}(R)\label{eqn:10}
\end{align}

for $i,j \in \{-1, 1\}$. By equation $(k, r, s)$, we mean equation $k$ with $i=r$ and $j=s$ substituted in. Solving for $a_{-j}(R)$ in (\ref{eqn:4},$i,-j$) gives us
\begin{align}
    a_{-j}(R) = \frac{c_{i}(T)a_{-j}(S)c_{i}(R) + b_{-j}(T)c_{i}(S)b_{j}(R) - d_{-i}(T)a_{j}(S)d_{i}(R)}{a_{-j}(T)c_{i}(S)} \label{eqn:11}
\end{align}
and solving for $a_{j}(R)$ in (\ref{eqn:5},$i,j$) gives us
\begin{align}
    a_j(R) = \frac{c_i(T)a_j(S)b_j(R)+b_j(T)c_i(S)c_{-i}(R)-b_{-j}(T)d_{-i}(S)d_i(R)}{c_i(T)b_j(S)}. \label{eqn:12}
\end{align}
Substituting (\ref{eqn:11},$i,j$) and (\ref{eqn:12},$i,j$) into (\ref{eqn:1},$i,j$) yields
\begin{align*}
    &a_j(T)a_j(S)d_i(R)\\&+ d_i(T)c_i(S)\left(\frac{c_{i}(T)a_{-j}(S)c_{i}(R) + b_{-j}(T)c_{i}(S)b_{j}(R) - d_{-i}(T)a_{j}(S)d_{i}(R)}{a_{-j}(T)c_{i}(S)}\right)\\
    &= c_i(T)d_i(S)\left(\frac{c_i(T)a_j(S)b_j(R)+b_j(T)c_i(S)c_{-i}(R)-b_{-j}(T)d_{-i}(S)d_i(R)}{c_i(T)b_j(S)}\right)\\&+b_{-j}(T)b_{-j}(S)d_i(R).
\end{align*}
Rearranging this equation gives us
\begin{align}
&b_j(R)[d_i(T)b_{-j}(T)c_i(S)b_j(S) - d_i(S)c_i(T)a_j(S)a_{-j}(T)] = -c_{i}(R)[d_i(T)b_j(S)c_i(T)a_{-j}(S)]\nonumber\\
&+c_{-i}(R)[d_i(S)b_j(T)c_i(S)a_{-j}(T)] + d_i(R)[a_{-j}(T)b_{-j}(T)[b_j(S)b_{-j}(S) - d_i(S)d_{-i}(S)]\nonumber\\
&+a_j(S)b_j(S)[d_i(T)d_{-i}(T) - a_j(T)a_{-j}(T)]].\label{eqn:13}
\end{align}
 Equation (\ref{eqn:7},$i,j$) implies that
\begin{align}
    c_{-i}(R) &= c_i(R)\frac{c_i(T)c_{-i}(S)}{c_{-i}(T)c_i(S)}.\label{eqn:14}
\end{align}
Substituting this expression for $c_{-i}(R)$ into (\ref{eqn:13},$i,j$) gives us
\begin{align}
    &b_j(R)[d_i(T)b_{-j}(T)c_i(S)b_j(S) - d_i(S)c_i(T)a_j(S)a_{-j}(T)]\nonumber\\
&=c_{i}(R)\left[\frac{-c_{-i}(T)d_i(T)b_j(S)c_i(T)a_{-j}(S)+c_i(T)c_{-i}(S)d_i(S)b_j(T)a_{-j}(T)}{c_{-i}(T)}\right]\nonumber\\&+ d_i(R)[a_{-j}(T)b_{-j}(T)[b_j(S)b_{-j}(S) - d_i(S)d_{-i}(S)]\nonumber\\
&+a_j(S)b_j(S)[d_i(T)d_{-i}(T) - a_j(T)a_{-j}(T)]]\label{eqn:15}
\end{align}
We repeat the above process for equations \ref{eqn:2} and \ref{eqn:3}. Solving for $a_{-j}(R)$ in (\ref{eqn:2},$i,-j$) gives us
\begin{align}
    a_{-j}(R) &= \frac{d_i(T)b_{-j}(S)c_i(R) + a_{-j}(T)d_i(S)b_{j}(R) - c_{-i}(T)b_{j}(S)d_i(R)}{b_{-j}(T)d_i(S)}\label{eqn:16}
\end{align}
and solving for $a_{j}(R)$ in (\ref{eqn:3},$i,j$) gives us
\begin{align}
    a_{j}(R) &= \frac{d_i(T)b_{j}(S)b_{j}(R) + a_{j}(T)d_i(S)c_{-i}(R) - a_{-j}(T)c_{-i}(S)d_i(R)}{d_i(T)a_{j}(S)}.\label{eqn:17}
\end{align}
Substituting (\ref{eqn:16},$i,j$) and (\ref{eqn:17},$i,j$) into (\ref{eqn:1},$i,-j$) yields
\begin{align*}
    &a_{-j}(T)a_{-j}(S)d_i(R) \\&+d_i(T)c_i(S)\left(\frac{d_i(T)b_{j}(S)b_{j}(R) + a_{j}(T)d_i(S)c_{-i}(R) - a_{-j}(T)c_{-i}(S)d_i(R)}{d_i(T)a_{j}(S)}\right)\nonumber\\
    &=c_i(T)d_i(S)\left(\frac{d_i(T)b_{-j}(S)c_i(R) + a_{-j}(T)d_i(S)b_{j}(R) - c_{-i}(T)b_{j}(S)d_i(R)}{b_{-j}(T)d_i(S)}\right)\nonumber\\
    &+b_{j}(T)b_{j}(S)d_i(R).
\end{align*}
Rearranging this equation gives us
\begin{align}
   &b_j(R)[c_i(S)d_i(T)b_j(S)b_{-j}(T) - c_i(T)a_{-j}(T)d_i(S)a_j(S)] = c_i(R)[c_i(T)d_i(T)b_{-j}(S)a_j(S)]\nonumber\\
   &-c_{-i}(R)[c_i(S)a_j(T)d_i(S)b_{-j}(T)]+d_i(R)[a_{-j}(T)b_{-j}(T)[c_i(S)c_{-i}(S) - a_j(S)a_{-j}(S)]\nonumber\\&+ a_j(S)b_j(S)[b_j(T)b_{-j}(T) - c_i(T)c_{-i}(T)]].\label{eqn:18}
\end{align}
Substituting (\ref{eqn:14},$i,j$) into (\ref{eqn:18},$i,j$) yields
\begin{align}
       &b_j(R)[c_i(S)d_i(T)b_j(S)b_{-j}(T) - c_i(T)a_{-j}(T)d_i(S)a_j(S)]\nonumber\\&=c_i(R)\left[\frac{c_{-i}(T)c_i(T)d_i(T)b_{-j}(S)a_j(S)-c_i(T)c_{-i}(S)a_j(T)d_i(S)b_{-j}(T)}{c_{-i}(T)}\right]\nonumber\\&+d_i(R)[a_{-j}(T)b_{-j}(T)[c_i(S)c_{-i}(S) - a_j(S)a_{-j}(S)]\nonumber\\&+ a_j(S)b_j(S)[b_j(T)b_{-j}(T) - c_i(T)c_{-i}(T)]].\label{eqn:19}
\end{align}
Since the left hand sides of (\ref{eqn:15},$i,j$) and (\ref{eqn:19},$i,j$) are identical, setting their right hand sides equal to each other gives us
\begin{align}
    &\frac{c_i(R)c_i(T)}{c_{-i}(T)}[c_{-i}(T)d_i(T)[b_{-j}(S)a_j(S)+b_j(S)a_{-j}(S)]\nonumber\\
    &- c_{-i}(S)d_i(S)[a_j(T)b_{-j}(T)+b_j(T)a_{-j}(T)]]\nonumber\\
    &= d_i(R)[a_{-j}(T)b_{-j}(T)F(S)-a_j(S)b_j(S)F(T)] \label{eqn:20}
\end{align}
where
\[F(\psi) = a_1(\psi)a_{-1}(\psi) + b_1(\psi)b_{-1}(\psi) - c_1(\psi)c_{-1}(\psi) - d_1(\psi)d_{-1}(\psi). \]
Now we repeat everything above, except with equation (\ref{eqn:6}) in place of equation (\ref{eqn:1}), and the variables $b$ instead of $a$. We can rewrite equation (\ref{eqn:2},$i,j$) as
\begin{align}
    d_i(S)b_{-j}(R) &= \frac{b_{j}(T)d_i(S)a_{j}(R) + c_{-i}(T)b_{-j}(S)d_i(R) - d_i(T)b_{j}(S)c_i(R)}{a_{j}(T)}\label{eqn:21}
\end{align}
and equation (\ref{eqn:5},$-i,j$) as
\begin{align}
    c_{-i}(T)b_{j}(R) = \frac{c_{-i}(T)b_{j}(S)a_{j}(R) + b_{-j}(T)d_i(S)d_{-i}(R) - b_{j}(T)c_{-i}(S)c_{i}(R)}{a_{j}(S)}.\label{eqn:22}
\end{align}
Substituting equations (\ref{eqn:21},$i,j$) and (\ref{eqn:22},$i,j$) into (\ref{eqn:6},$i,-j$) yields
\begin{align}
    &a_{j}(R)[c_i(S)c_{-i}(T)b_{j}(S)a_{j}(T) - d_{-i}(T)b_{j}(T)d_i(S)a_{j}(S)]\nonumber\\
    &=c_i(R)[a_{j}(T)b_{j}(T)[c_i(S)c_{-i}(S) - a_{-j}(S)a_{j}(S)] +a_{j}(S)b_{j}(S)[a_{-j}(T)a_{j}(T) - d_i(T)d_{-i}(T)]]\nonumber\\
    &+ d_i(R)[d_{-i}(T)c_{-i}(T)b_{-j}(S)a_{j}(S)] - d_{-i}(R)[c_i(S)b_{-j}(T)d_i(S)a_{j}(T)].\label{eqn:23}
\end{align}
Equation (\ref{eqn:8},$i$) implies that
\begin{align}
    d_{-i}(R) = d_i(R)\frac{d_{-i}(T)c_{-i}(S)}{d_i(T)c_i(S)}.\label{eqn:24}
\end{align}
Substituting this expression into (\ref{eqn:23},$i,j$) gives us
\begin{align}
    &a_{j}(R)[c_i(S)c_{-i}(T)b_{j}(S)a_{j}(T) - d_{-i}(T)b_{j}(T)d_i(S)a_{j}(S)]\nonumber\\
    &=c_i(R)[a_{j}(T)b_{j}(T)[c_i(S)c_{-i}(S) - a_{-j}(S)a_{j}(S)] +a_{j}(S)b_{j}(S)[a_{-j}(T)a_{j}(T) - d_i(T)d_{-i}(T)]]\nonumber\\
    &+ \frac{d_i(R)}{d_i(T)}[d_i(T)d_{-i}(T)c_{-i}(T)b_{-j}(S)a_j(S)-d_{-i}(T)c_{-i}(S)b_{-j}(T)d_i(S)a_j(T)].\label{eqn:25}
\end{align}
We can rewrite equation (\ref{eqn:3},$-i,j$) as
\begin{align}
    d_{-i}(T)b_j(R) = \frac{d_{-i}(T)a_j(S)a_j(R) + a_{-j}(T)c_{i}(S)d_{-i}(R) - a_j(T)d_{-i}(S)c_{i}(R)}{b_j(S)}\label{eqn:26}
\end{align}
and equation (\ref{eqn:4},$i,j$) as
\begin{align}
    c_i(S)b_{-j}(R) = \frac{a_j(T)c_i(S)a_j(R) + d_{-i}(T)a_{-j}(S)d_i(R) - c_i(T)a_j(S)c_i(R)}{b_j(T)}\label{eqn:27}.
\end{align}
Substituting equations (\ref{eqn:26},$i,j$) and (\ref{eqn:27},$i,j$) into (\ref{eqn:6},$i,j$) gives us
\begin{align}
    &a_j(R)[c_{-i}(T)a_j(T)c_i(S)b_j(S) - d_i(S)d_{-i}(T)a_j(S)b_j(T)]\nonumber\\
    &= c_i(R)[a_j(S)b_j(S)[c_i(T)c_{-i}(T) - b_j(T)b_{-j}(T)] + a_j(T)b_j(T)[b_j(S)b_{-j}(S)-d_i(S)d_{-i}(S)]]\nonumber\\
    &+ d_{-i}(R)[d_i(S)a_{-j}(T)c_i(S)b_j(T)] - d_i(R)[c_{-i}(T)d_{-i}(T)a_{-j}(S)b_j(S)].\label{eqn:28}
\end{align}
Substituting equation (\ref{eqn:24},$i,j$) into (\ref{eqn:28},$i,j$) gives us
\begin{align}
    &a_j(R)[c_{-i}(T)a_j(T)c_i(S)b_j(S) - d_i(S)d_{-i}(T)a_j(S)b_j(T)]\nonumber\\
    &= c_i(R)[a_j(S)b_j(S)[c_i(T)c_{-i}(T) - b_j(T)b_{-j}(T)] + a_j(T)b_j(T)[b_j(S)b_{-j}(S)-d_i(S)d_{-i}(S)]]\nonumber\\
    &+ \frac{d_i(R)}{d_i(T)}[d_{-i}(T)c_{-i}(S)d_i(S)a_{-j}(T)b_j(T) - d_i(T)c_{-i}(T)d_{-i}(T)a_{-j}(S)b_j(S)].\label{eqn:29}
\end{align}
Since equations (\ref{eqn:25},$i,j$) and (\ref{eqn:29},$i,j$) have the same left hand side, we have
\begin{align}
    &c_i(R)[a_j(T)b_j(T)F(S) - a_j(S)b_j(S)F(T)]\nonumber\\
    &= \frac{d_i(R)d_{-i}(T)}{d_i(T)}[c_{-i}(T)d_i(T)[b_{-j}(S)a_j(S) + a_{-j}(S)b_j(S)] - c_{-i}(S)d_i(S)[b_{-j}(T)a_j(T) + a_{-j}(T)b_j(T)]].\label{eqn:30}
\end{align}
Let
\begin{align*}
    &G_i(S,T) = [ c_{-i}(T)d_i(T)[b_{-1}(S)a_1(S)+ a_{-1}(S)b_1(S)]- c_{-i}(S)d_i(S)[b_{-1}(T)a_1(T)+a_{-1}(T)b_1(T)]].
\end{align*}
Our calculations above allow us to prove the following condition.
\begin{theorem}
Let $S, T \in \End(V\otimes V)$. Suppose that there exists $R \in \text{End}(V\otimes V)$ such that $c_{-1}(R),c_1(R),d_{-1}(R),d_1(R)$ are nonzero and $[[R,S,T]] = 0$. Then, we have \begin{align}
    a_1(T)b_1(T)F(S) &= a_{-1}(T)b_{-1}(T)F(S)\label{condition1}\\
    a_1(S)b_1(S)F(T) &= a_{-1}(S)b_{-1}(S)F(T)\label{condition2}\\
    \frac{c_i(T)d_{-i}(T)}{c_{-i}(T)d_i(T)}G_i(S,T)^2 &= [a_1(T)b_1(T)F(S) - a_1(S)b_1(S)F(T)]^2\label{condition3}\\
    \frac{c_1(T)c_{-1}(S)}{c_{-1}(T)c_1(S)} &= \frac{d_1(T)d_{-1}(S)}{d_{-1}(T)d_1(S)}.\label{condition4}
\end{align}
\end{theorem}
\begin{proof}
We write the system of equations given by $(\ref{eqn:20})$ and $(\ref{eqn:30})$ as the matrix equation
\[\begin{pmatrix}
\frac{c_i(T)}{c_{-i}(T)}G_i(S,T) && - \alpha_1(S,T)\\
\frac{c_i(T)}{c_{-i}(T)}G_i(S,T) && -\alpha_2(S,T)\\
\beta_1(S,T) && -\frac{d_{-i}(T)}{d_i(T)}G_i(S,T)\\
\beta_2(S,T) && -\frac{d_{-i}(T)}{d_i(T)}G_i(S,T)
\end{pmatrix}\begin{pmatrix}
c_i(R)\\d_i(R)
\end{pmatrix} = 0,
\]
where
\[\alpha_j(S,T) = a_j(T)b_j(T)F(S) - a_{-j}(S)b_{-j}(S)F(T) \]
and
\[\beta_j(S,T) = a_j(T)b_j(T)F(S) - a_j(S)b_j(S)F(T).\]
First we observe that if $G_i(S,T) = 0$, then we must have $\alpha_1(S,T) = \alpha_2(S,T) = \beta_1(S,T) = \beta_2(S,T) = 0$ in order for both $c_i(R)$ and $d_i(R)$ to be nonzero. It is easy to see that this implies equations (\ref{condition1}), (\ref{condition2}), (\ref{condition3}). Now suppose that $G_i(S,T) \neq 0$. A necessary condition for one of $c_{-1}(R),c_1(R),d_{-1}(R),$ and $d_1(R)$ to be nonzero is for all 6 of the 2 by 2 minors of the above matrix to vanish for $i = 1,-1$. The vanishing of the minors of the rows $(1,2),(3,4),(1,3)$ can be rewritten as
\begin{align}
    \alpha_1(S,T) &= \alpha_2(S,T)\label{eqn:35}\\
    \beta_1(S,T) &= \beta_2(S,T)\label{eqn:36}\\
    \frac{c_i(T)d_{-i}(T)}{c_{-i}(T)d_i(T)}G_i(S,T)^2 &= \alpha_1(S,T)\beta_1(S,T).\label{eqn:37}
\end{align}
By combining equations (\ref{eqn:35}) and (\ref{eqn:36}), we obtain conditions (\ref{condition1}) and (\ref{condition2}). Then, using equations (\ref{eqn:37}) and (\ref{condition1}),we get
\begin{align*}
    \frac{c_i(T)d_{-i}(T)}{c_{-i}(T)d_i(T)}G_i(S,T)^2 &= [a_1(T)b_1(T)F(S) - a_2(S)b_2(S)F(T)][a_1(T)b_1(T)F(S) - a_1(S)b_1(S)F(T)]\\
    &= [a_1(T)b_1(T)F(S) - a_1(S)b_1(S)F(T)][a_1(T)b_1(T)F(S) - a_1(S)b_1(S)F(T)],
\end{align*}
which shows that (\ref{condition3}) holds. So in all cases, conditions (\ref{condition1}), (\ref{condition2}), and (\ref{condition3}) hold.

To get the last equation, we combine equations (\ref{eqn:7},$i$) and (\ref{eqn:10},$i$).
\end{proof}
\section{Future Directions}
A main theme of this paper was interpreting properties of lattice models in terms of differential forms. We have only scratched the surface of the problems this technique can be applied to. In particular, we restricted our attention to lattice models without weights assigned at each vertex. A first step toward understanding weighted partition functions in this viewpoint would be to find an interpretation of the Yang-Baxter equation using differential forms.

Another potential area of further investigation is to describe interesting families of weights for which the eight-vertex Yang-Baxter equation holds. This has been explored in the case $c_1 = c_{-1}$ and $d_1 = d_{-1}$ by Cuerno et al \cite{cuerno}. It is natural to ask if the weights can be generalized to families where $c_1 \neq c_{-1}$ or $d_1 \neq d_{-1}$. 
\section{Acknowledgements}
This research was conducted at the 2020 University of Minnesota Twin Cities REU with the support of the NSF RTG grant DMS-1745638. I would like to thank Ben Brubaker and Claire Frechette for their mentorship and support.
\bibliographystyle{amsplain}
\bibliography{references}
\end{document}